\documentclass{amsart}

\usepackage{amsmath}
\usepackage{amsthm}
\usepackage{amssymb}
\usepackage{txfonts}

\theoremstyle{plain}
\newtheorem{definition}{Definition}
\newtheorem{thm}[definition]{Theorem}

\newtheorem{lem}[definition]{Lemma}
\newtheorem{cor}[definition]{Corollary}

\newtheorem{remark}[definition]{Remark}

\def\F#1#2#3#4{F\biggl(\genfrac..{0pt}{}{{#1},\,{#2}}{#3}\,;#4\biggr)}
\def\p#1#2#3#4{\phi\biggl(\genfrac..{0pt}{}{{#1},\,{#2}}{#3}\,;q,\,#4\biggr)}
\def\P#1#2#3#4#5#6#7#8{{}_{4}\phi_{3}\biggl(\genfrac..{0pt}{}{{#1},\,{#2},\,{#3},\,{#4}}{{#5},\,{#6},\,{#7}}\,;q,\,#8\biggr)}
\def\f#1#2#3#4{f\biggl(\genfrac..{0pt}{}{{#1},\,{#2}}{#3}\,;#4\biggr)}
\def\g#1#2#3#4{g\biggl(\genfrac..{0pt}{}{{#1},\,{#2}}{#3}\,;q,\,#4\biggr)}
\def\y#1#2#3#4#5{y_{#1}\biggl(\genfrac..{0pt}{}{{#2},\,{#3}}{#4}\,;#5\biggr)}
\def\c#1#2#3#4#5#6#7{\biggl(\genfrac..{0pt}{}{{#1},\,{#2}}{#3}\,;\genfrac..{0pt}{}{{#4},\,{#5}}{#6}\,;q,\,#7\biggr)}
\def\Y#1#2#3#4#5#6#7{Y\biggl(\genfrac..{0pt}{}{{#1},\,{#2}}{#3}\,;\genfrac..{0pt}{}{{#4},\,{#5}}{#6}\,;#7\biggr)}
\def\Z#1#2#3#4#5#6#7{\tilde{Y}\biggl(\genfrac..{0pt}{}{{#1},\,{#2}}{#3}\,;\genfrac..{0pt}{}{{#4},\,{#5}}{#6}\,;#7\biggr)}

\def\G{\Gamma}
\def\vec#1{\mbox{\boldmath $#1$}}

\makeatletter
\@addtoreset{equation}{section}

\makeatother

\begin{document}
\title[{\bf Three Term Relations for Basic Hypergeometric Series}]{{\bf Three Term Relations for Basic Hypergeometric Series}}
\author[Yuka Suzuki]{Yuka Suzuki}
\date{}
\subjclass[2010]{33D15.}
\keywords{Basic hypergeometric series; $q\mbox{-}$hypergeometric series; Three term relation; Contiguous relation.}

\maketitle

\begin{abstract}
Any three basic hypergeometric series ${}_{2} \phi_{1}$ whose respective parameters $(a, b, c)$ 
differ by integer powers of the base $q$ satisfy a linear relation 
with coefficients which are rational functions of $a,\ b,\ c,\ q$ and the variable $x$. 
These relations are called three term relations for the basic hypergeometric series ${}_{2} \phi_{1}$. 

This paper gives explicit expressions for the coefficients of these three term relations. 
\end{abstract}

\section{\bf{Introduction}}
In this paper, we give explicit expressions for the coefficients of 
three term relations for the basic hypergeometric series ${}_{2} \phi_{1}$, 
which are $q$-analogues of Ebisu's \cite{Eb1} results for Gauss hypergeometric series.

The basic hypergeometric series ${}_{p + 1} \phi_{p}$ is defined by 
\begin{align*}
{}_{p + 1}\phi_{p}\biggl(\genfrac..{0pt}{}{a_{0},\,a_{1}, \dotsc, a_{p}}{b_{1}, \dotsc, b_{p}}\,; q, x\biggr) 
&= {}_{p + 1}\phi_{p} (a_{0}, a_{1}, \ldots, a_{p}\,; b_{1}, \ldots, b_{p}\,; q, x) \\
&:= \sum_{n = 0}^{\infty} \frac{(a_{0}; q)_{n} (a_{1}; q)_{n} \dotsm (a_{p}; q)_{n}}{(q; q)_{n} (b_{1}; q)_{n} \dotsm (b_{p}; q)_{n}} x^{n}, \nonumber 
\end{align*}
where $q$ is a fixed complex number satisfying $\lvert q \rvert < 1$, 
and $(a; q)_{n}$ denotes the $q\mbox{-}$shifted factorial defined by 
$(a; q)_{n} := (a; q)_{\infty} / (a q^{n}; q)_{\infty}$, 
$(a; q)_{\infty} := \prod_{n = 0}^{\infty} (1 - a q^{n})$. 
It is assumed that $b_{1}, \dotsc, b_{p}$ are not $1$ or negative integer powers of $q$. 
For convenience of notation, we denote $(a; q)_{n},\ (a; q)_{\infty}$ and ${}_{2}\phi_{1}$ 
by $(a)_{n},\ (a)_{\infty}$ and $\phi$, respectively. 
Also, we set $a = q^{\alpha},\ b = q^{\beta}$ and $c = q^{\gamma}$.

The series $\phi (a, b; c; q, x)$ is the $q\mbox{-}$analogue of Gauss hypergeometric series 
$F (\alpha, \beta; \gamma; x)$. 
It is known that for any triples of integers $(k, l, m)$ and $(k', l', m')$, 
three hypergeometric series 
\begin{align*}
\F{\alpha + k}{\beta + l}{\gamma + m}{x}, \quad \F{\alpha + k'}{\beta + l'}{\gamma + m'}{x}, 
\quad \F{\alpha}{\beta}{\gamma}{x}
\end{align*}
satisfy a linear relation with coefficients which are rational functions of $\alpha,\ \beta,\ \gamma$ and $x$. 
This relation is called the three term relation for $F$. 
Gauss obtained three term relations for $F$ for the cases of 
\begin{align*}
(k, l, m), (k', l', m') \in \left\{(1, 0, 0), (-1, 0, 0), (0, 1, 0), (0, -1, 0), (0, 0, 1), (0, 0, -1) \right\}, 
\end{align*}
where $(k, l, m) \neq (k', l', m')$. 
That is, there are $\binom{6}{2} = 15$ pairs of $(k, l, m)$ and $(k', l', m')$. 
See \cite[Chapter~$4$, p.$71$]{Rain} for fifteen relations due to Gauss. 
Ebisu \cite{Eb1} considered three term relations for $F$ for the cases of 
$(k, l, m) \in \mathbb{Z}^{3}$ and $(k', l', m') = (1, 1, 1)$ : 
\begin{align}\label{3tr_F}
\F{\alpha + k}{\beta + l}{\gamma + m}{x} 
= P_{1} \cdot \F{\alpha + 1}{\beta + 1}{\gamma + 1}{x} + P_{2} \cdot \F{\alpha}{\beta}{\gamma}{x}, 
\end{align}
and showed that the coefficients $P_{1}$ and $P_{2}$ can be expressed as 
sums of products of two hypergeometric series. 

Similarly, for any triples of integers $(k, l, m)$ and $(k', l', m')$, 
three basic hypergeometric series 
\begin{align*}
\p{a q^{k}}{b q^{l}}{c q^{m}}{x}, \quad 
\p{a q^{k'}}{b q^{l'}}{c q^{m'}}{x}, \quad \p{a}{b}{c}{x}
\end{align*}
satisfy a linear relation with coefficients which are rational functions of $a,\ b,\ c,\ q$ and $x$. 
We call this relation the three term relation for $\phi$. 
Fifteen three term relations for $\phi$ corresponding to 
Gauss's three term relations for $F$ were obtained by Kim, Rathie and Choi \cite{KRC}.

In this paper, we consider the following three term relation for $\phi$ corresponding to $(\ref{3tr_F})$. 
\begin{align}\label{3tr}
\p{a q^{k}}{b q^{l}}{c q^{m}}{x} 
= Q \cdot \p{a q}{b q}{c q}{x} + R \cdot \p{a}{b}{c}{x}. 
\end{align}
We show that the pair $(Q, R)$ of functions, rational in $a,\ b,\ c,\ q$ and $x$, 
are uniquely determined by $(k, l, m)$, 
and give explicit expressions for $Q$ and $R$. 
It should be noted that if we use these expressions, then we can obtain more general three term relations 
\begin{align*}
\p{a q^{k}}{b q^{l}}{c q^{m}}{x} 
= Q' \cdot \p{a q^{k'}}{b q^{l'}}{c q^{m'}}{x} + R' \cdot \p{a}{b}{c}{x}
\end{align*}
by eliminating $\phi (a q, b q; c q; x)$ from three term relations $(\ref{3tr})$ 
for $(k, l, m)$ and $(k', l', m')$. 

Throughout this paper, unless otherwise stated explicitly, we assume that 
\begin{align*}
{\rm E1} : c,\, \frac{a}{b} \notin q^{\mathbb{Z}}, \quad 
{\rm E2} : a,\, b,\, \frac{c}{a},\, \frac{c}{b} \notin q^{\mathbb{Z}}. 
\end{align*}

\subsection{The coefficients of three term relations for $\phi$}\mbox{} \\

Our main results are the following three theorems. 
Remark that we only need to consider $(\ref{3tr})$ for the cases of $k \leq l$ 
since $\phi (a, b; c; q, x)$ is symmetric in $a$ and $b$. 

The following theorem says about the uniqueness of the pair $(Q, R)$ satisfying $(\ref{3tr})$. 
\begin{thm}\label{uniqueness}
For any triple of integers $(k, l, m)$, the pair $(Q, R)$ of rational functions of $a,\ b,\ c,\ q$ and $x$, 
satisfying the three term relation $(\ref{3tr})$, are uniquely determined by $(k, l, m)$. 
\end{thm}
The following two theorems give explicit expressions for $Q$ and $R$. 
Selectively using these theorems will be helpful when we calculate $Q$ and $R$ (see Remark~$\ref{rem}$). 
\begin{thm}\label{main}
For any integers $k,\ l$ and $m$ with $k \leq l$, the coefficients of $(\ref{3tr})$ can be expressed as 
\begin{align*}
Q &= Q (k, l, m) 
= - \frac{(1 - a) (1 - b) c}{(q - c) (1 - c)} 
\frac{x^{1 - \max \left\{m, 0 \right\}}}{(a b q x / c)_{\max \left\{k + l - m, 0 \right\} - 1}} 
P \c{k}{l}{m}{a}{b}{c}{x}, \allowdisplaybreaks \\
R &= R (k, l, m)
= - \frac{x^{- \max \left\{m - 1, 0 \right\}}}{(a b q x / c)_{\max \left\{k + l - m - 1, 0 \right\}}} 
P \c{k - 1}{l - 1}{m - 1}{a q}{b q}{c q}{x}, 
\end{align*}
where $P$ is a polynomial in $x$ defined by 
\begin{align*}
P = P \c{k}{l}{m}{a}{b}{c}{x} := 
\begin{cases}
\displaystyle\sum_{n = 0}^{l - 1} \left(A_{n} - B_{n - m} \right) x^{n} 
& (m \geq 0,\, k + l - m \geq 0), \\
\displaystyle\sum_{n = 0}^{m - k - 1} \left(\tilde{A}_{n} - \tilde{B}_{n - m} \right) x^{n} 
& (m \geq 0,\, k + l - m < 0), \\
\displaystyle\sum_{n = 0}^{l - m - 1} \left(A_{n + m} - B_{n} \right) x^{n} 
& (m < 0,\, k + l - m \geq 0), \\
\displaystyle\sum_{n = 0}^{- k - 1} \left(\tilde{A}_{n + m} - \tilde{B}_{n} \right) x^{n} 
& (m < 0,\, k + l - m < 0). 
\end{cases}
\end{align*}
Here, $A_{n},\ B_{n},\ \tilde{A}_{n}$ and $\tilde{B}_{n}$ are rational functions of $a,\ b,\ c$ and $q$ 
(see Definition~$\ref{def:P,A,B}$ for details). 
\end{thm}
\begin{thm}\label{main2}
For any integers $k,\ l$ and $m$ with $k \leq l$, the coefficients of $(\ref{3tr})$ can be expressed as 
\begin{align*}
Q &= Q (k, l, m) 
= - \lambda \frac{(1 - a) (1 - b) c}{(q - c) (1 - c)} 
\frac{x^{1 - \max \left\{m, 0 \right\}}}{(a b q x / c)_{\max \left\{k + l - m, 0 \right\} - 1}} 
\tilde{P} \c{k}{l}{m}{a}{b}{c}{x}, \allowdisplaybreaks \\
R &= R (k, l, m)
= - \lambda' \frac{x^{- \max \left\{m - 1, 0 \right\}}}{(a b q x / c)_{\max \left\{k + l - m - 1, 0 \right\}}} 
\tilde{P} \c{k - 1}{l - 1}{m - 1}{a q}{b q}{c q}{x}, 
\end{align*}
where 
\begin{align*}
\lambda := (-1)^{k + l - m - 1 + M} 
q^{\left\{k (k - 1) + l (l - 1) - m (m - 1) + M \left(M - 1 \right) \right\} / 2}\, 
\frac{(q - c) a^{k} b^{l}}{(b - a) c^{m}} \left(\frac{a b}{c} \right)^{M} \frac{(c)_{m}}{(a)_{k} (b)_{l}} 
\end{align*}
with $M := \max\left\{k + l - m, 0 \right\}$, 
\begin{align*}
\lambda' := (-1)^{k + l - m + M'} 
q^{\left\{k (k - 1) + l (l - 1) - m (m - 1) + M' \left(M' + 1 \right) \right\} / 2}\, 
\frac{a^{k} b^{l}}{(b - a) c^{m}} \left(\frac{a b}{c} \right)^{M' - 1} \frac{(c)_{m}}{(a q)_{k - 1} (b q)_{l - 1}} 
\end{align*}
with $M' := \max\left\{k + l - m - 1, 0 \right\}$ 
and $\tilde{P}$ is a polynomial in $x$ defined by 
\begin{align*}
\tilde{P} = 
\tilde{P} \c{k}{l}{m}{a}{b}{c}{x} := 
\begin{cases}
\displaystyle\sum_{n = 0}^{l - 1} \left(C_{n} - D_{n + k - l} \right) x^{l - 1 - n} 
& (m \geq 0,\, k + l - m \geq 0), \\
\displaystyle\sum_{n = 0}^{m - k - 1} \left(\tilde{C}_{n} - \tilde{D}_{n + k - l} \right) x^{m - k - 1 - n} 
& (m \geq 0,\, k + l - m < 0), \\
\displaystyle\sum_{n = 0}^{l - m - 1} \left(C_{n} - D_{n + k - l} \right) x^{l - m - 1 - n} 
& (m < 0,\, k + l - m \geq 0), \\
\displaystyle\sum_{n = 0}^{- k - 1} \left(\tilde{C}_{n} - \tilde{D}_{n + k - l} \right) x^{- k - 1 - n} 
& (m < 0,\, k + l - m < 0). 
\end{cases}
\end{align*}
Here, $C_{n},\ D_{n},\ \tilde{C}_{n}$ and $\tilde{D}_{n}$ are rational functions of $a,\ b,\ c$ and $q$ 
(see Definition~$\ref{def:til(P),C,D}$ for details). 
\end{thm}
From Theorems~$\ref{uniqueness}\mbox{--}\ref{main2}$, we obtain the following corollary 
which gives information about the polynomials $P$ and $\tilde{P}$. 
\begin{cor}\label{saikouji}
The polynomial $P$ in Theorem~$\ref{main}$ and the polynomial $\tilde{P}$ in Theorem~$\ref{main2}$ 
satisfy 
\begin{align}\label{P,til(P)}
P \c{k}{l}{m}{a}{b}{c}{x} = \lambda \cdot \tilde{P} \c{k}{l}{m}{a}{b}{c}{x}, 
\end{align}
where $\lambda$ is defined as in Theorem~$\ref{main2}$. 
Therefore, the degree of $P$ equals the degree of $\tilde{P}$. 
Let $d := \max \left\{m, 0 \right\} + \max \left\{k + l - m, 0 \right\} - k - 1$. 
Then, for any integers $k,\ l$ and $m$ with $k \leq l$, 
the degree of the polynomial $\tilde{P}$ is no more than $d$, 
and the coefficient of $x^{d}$ in $\tilde{P}$ is equal to
\begin{align*}
\begin{cases}
\,\displaystyle\left(\frac{c}{a b} \right)^{k} q^{k (m - k - l + 1)} 
\left\{a^{m - 2k} (a)_{k} (a q / c)_{k - m} - b^{m - 2k} (b)_{k} (b q / c)_{k - m} \right\} & (k = l), \\
\,\displaystyle\left(\frac{c}{a b} \right)^{k} q^{k (m - k - l + 1)} a^{m - k - l} 
\frac{(a)_{k} (a q / c)_{k - m}}{(a q / b)_{k - l}} & (k < l). 
\end{cases}
\end{align*}
\end{cor}
\begin{remark}\label{rem}
Let us call $P$ and $\tilde{P}$ the polynomial parts of $Q$ and $R$. 
Then, it follows from $(\ref{P,til(P)})$ and the definitions of $P$ and $\tilde{P}$ that 
Theorem~$\ref{main}$ is useful for calculating the coefficients of lower-order terms in 
the polynomial parts of $Q$ and $R$, 
while Theorem~$\ref{main2}$ is useful for calculating the coefficients of higher-order terms in 
the polynomial parts. 
\end{remark}

From either Theorem~$\ref{main}$ or Theorem~$\ref{main2}$, we obtain the following corollary 
which gives a relation between $Q$ and $R$. 
\begin{cor}\label{R=Q'}
The coefficients of the three term relation $(\ref{3tr})$ satisfy 
\begin{align*}
Q (k - 1, l - 1, m - 1) \bigg| _{(a, b, c) \mapsto (a q, b q, c q)} 
= \frac{(1 - a q) (1 - b q) x (c - a b q x)}{(1 - c) (1 - c q)} R (k, l, m). 
\end{align*}
\end{cor}

\section{\bf{The coefficients of three term relations for $\phi$}}
In this section, we prove Theorems~$\ref{uniqueness}\mbox{--}\ref{main2}$ and Corollaries~$\ref{saikouji}, \ref{R=Q'}$. 
First, we show that Theorem~$\ref{uniqueness}$. 
Next, we prove Theorems~$\ref{main}$ and $\ref{main2}$. 
For this purpose, we introduce the series 
\begin{align*}
\f{a}{b}{c}{x} := (1 - q)^{\gamma - \alpha - \beta + 1} \frac{(q)_{\infty} (c)_{\infty}}{(a)_{\infty} (b)_{\infty}} 
\p{a}{b}{c}{x}, 
\end{align*}
and rewrite the three term relation $(\ref{3tr})$ as 
\begin{align}\label{3tr'}
\f{a q^{k}}{b q^{l}}{c q^{m}}{x} 
= \tilde{Q} \cdot \f{a q}{b q}{c q}{x} + \tilde{R} \cdot \f{a}{b}{c}{x}. 
\end{align}
By comparing $(\ref{3tr})$ with $(\ref{3tr'})$, 
we find that $Q$ and $R$ can be expressed by $\tilde{Q}$ and $\tilde{R}$ as 
\begin{align}
Q &= (1 - q)^{k + l - m - 1} \frac{(c q)_{m - 1}}{(a q)_{k - 1} (b q)_{l - 1}} \tilde{Q}, \label{Q,tQ} \\
R &= (1 - q)^{k + l - m} \frac{(c)_{m}}{(a)_{k} (b)_{l}} \tilde{R}. \label{R,tR}
\end{align}
Thus we investigate $\tilde{Q}$ and $\tilde{R}$. 
We define four functions 
which are local solutions of a $q\mbox{-}$differential equation $E_{q} (a, b, c)$ 
defined by $L_{q} (a, b, c) y = 0$, where 
\begin{align*}
L_{q} (a, b, c) := x (c - a b q x) \Delta_{q}^{2} 
+ \left\{\!\frac{1 - c}{1 - q} + \frac{(1 - a) (1 - b) - (1 - a b q)}{1 - q} x \!\right\} \!\Delta_{q} 
- \frac{(1 - a) (1 - b)}{(1 - q)^{2}} 
\end{align*}
with 
\begin{align*}
\Delta_{q} y (x) = \frac{y (x) - y (q x)}{(1 - q) x}. 
\end{align*}
We also define first order $q\mbox{-}$differential operators called contiguity operators. 
By using the $q\mbox{-}$differential equation $E_{q} (a, b, c)$ and contiguity operators, 
we obtain linear equations with $\tilde{Q}$ and $\tilde{R}$, 
and by solving these equations, we express $\tilde{Q}$ and $\tilde{R}$ as ratios of infinite series. 
To express $\tilde{Q}$ and $\tilde{R}$ more explicitly, 
we introduce Heine's transformation formula for $\phi$ and Gasper's three formulas for basic hypergeometric series. 
By using these formulas, we obtain expressions for $\tilde{Q}$ and $\tilde{R}$ as rational functions, 
and complete the proof of Theorems~$\ref{main}$ and $\ref{main2}$. 
Finally, we obtain Corollaries~$\ref{saikouji}$ and $\ref{R=Q'}$ 
from Theorems~$\ref{uniqueness}\mbox{--}\ref{main2}$.

\subsection{Proof of Theorem~$\ref{uniqueness}$}\mbox{} \\

By using the method due to Vid\=unas \cite[3]{Vidunas} as reference, 
we show that Theorem~$\ref{uniqueness}$. For this purpose, 
we use the following summation formula which was independently discovered 
by Bailey \cite{Bailey} and Daum \cite{Daum} : 
\begin{align}\label{summation formula2}
\p{a}{b}{a q / b}{- \frac{q}{b}} 
= \frac{(-q)_{\infty} (a q ; q^{2})_{\infty} (a q^{2} / b^{2} ; q^{2})_{\infty}}{(-q / b)_{\infty} (a q / b)_{\infty}}, 
\quad \lvert q \rvert < \min\left\{1, \lvert b \rvert \right\}. 
\end{align}
This is a $q\mbox{-}$analogue of Kummer's summation formula 
\begin{align*}
\F{\alpha}{\beta}{\alpha + 1 - \beta}{-1} 
= \frac{\G (1 + (\alpha / 2)) \G (\alpha + 1 - \beta)}{\G (1 + \alpha) \G ((\alpha / 2) + 1 - \beta)}. 
\end{align*}
We also use the three term relation 
\begin{align}\label{(1,0,1)}
\phi (a q, b, c q) 
= \frac{(1 - b) (c - a b x)}{c - b} \phi (a q, b q, c q) 
- \frac{b (1 - c)}{c - b} \phi (a, b, c). 
\end{align}
This follows from Heine's \cite[Formulas $24.,\ 43.$ and $48.$]{Heine} three term relations 
\begin{align*}
\left( 1 - \frac{c}{a} \right) \phi (a q^{-1}, b, c) - \left(1 - \frac{c}{b} \right) \phi (a, b q^{-1}, c) 
&= \frac{c}{b} \left(1 - \frac{b}{a} \right) \left(1 - \frac{a b}{c q} x \right) \phi (a, b, c), \allowdisplaybreaks \\
\phi (a q, b, c) - \phi (a, b, c) &= \frac{a (1 - b) x}{1 - c} \phi (a q, b q, c q), \allowdisplaybreaks \\
\phi (a q, b q^{-1}, c) - \phi (a, b, c) &= \frac{(a q - b) x}{(1 - c) q} \phi (a q, b, c q). 
\end{align*}
Here, $\phi (a, b, c)$ denote $\phi (a, b; c; q, x)$. 

In order to prove by contradiction, let us assume that 
there are two distinct pairs of rational functions $(Q_{1}, R_{1})$ and $(Q_{2}, R_{2})$ 
satisfying $(\ref{3tr})$. 
Then, we have 
\begin{align*}
(Q_{1} - Q_{2}) \cdot \phi (a q, b q, c q) = (R_{2} - R_{1}) \cdot \phi (a, b, c). 
\end{align*}
This implies that $\phi (a q, b q, c q) / \phi (a, b, c)$ 
is a rational function of $a,\ b,\ c,\ q$ and $x$. 
Therefore, from $(\ref{(1,0,1)})$, 
\begin{align*}
\g{a}{b}{c}{x} := \p{a q}{b}{c q}{x} \bigg/ \p{a}{b}{c}{x} 
\end{align*}
is also a rational function of $a,\ b,\ c,\ q$ and $x$. 
However, from $(\ref{summation formula2})$, we have 
\begin{align*}
\g{a}{b}{a q / b}{- \frac{q}{b}} 
&= \p{a q}{b}{a q^{2} / b}{- \frac{q}{b}} \bigg/ \p{a}{b}{a q / b}{- \frac{q}{b}} \\
&= \left(1 - \frac{a q}{b} \right) 
\frac{(a q^{2} ; q^{2})_{\infty} (a q^{3} / b^{2} ; q^{2})_{\infty}}{(a q ; q^{2})_{\infty} (a q^{2} / b^{2} ; q^{2})_{\infty}}. 
\end{align*}
It turns out that $g (a) := g (a, b; a q / b; q, - q / b)$ is the rational function of $a$, 
while $g (a)$ has unbounded set of poles, so this is a contradiction. 
Thus Theorem~$\ref{uniqueness}$ is proved.

\subsection{Local solutions of $E_{q} (a, b, c)$}\mbox{} \\

We define four functions local solutions of the $q\mbox{-}$differential equation $E_{q} (a, b, c)$. 

Let $y_{i}$ ($i = 1, 2, 3, 4$) be the functions defined by
\begin{align*}
\y{1}{a}{b}{c}{x} = y_{1} (a, b; c; x) &:= \f{a}{b}{c}{x}, \allowdisplaybreaks \\
\y{2}{a}{b}{c}{x} = y_{2} (a, b; c; x) &:= x^{1 - \gamma} \f{a q / c}{b q / c}{q^{2} / c}{x}, \allowdisplaybreaks \\
\y{3}{a}{b}{c}{x} = y_{3} (a, b; c; x) &:= a^{\gamma - \alpha - \beta + 1} x^{-\alpha} 
\f{a}{a q / c}{a q / b}{\frac{c q}{a b x}}, \allowdisplaybreaks \\
\y{4}{a}{b}{c}{x} = y_{4} (a, b; c; x) &:= b^{\gamma - \alpha - \beta + 1} x^{-\beta} 
\f{b}{b q / c}{b q / a}{\frac{c q}{a b x}}. 
\end{align*}
We give two lemmas which are used in Section~$2. 4$ 
to obtain expressions for $\tilde{Q}$ and $\tilde{R}$. 
The following lemma follows from direct calculations. 
\begin{lem}\label{lem:1}
Let $y_{i} (a, b, c)$ denote $y_{i} (a, b; c; x)$. Then, we have 
\begin{align*}
\Delta_{q} y_{i} (a, b, c) = 
\begin{cases}
y_{i} (a q, b q, c q) & (i = 1, 2), \\
- \displaystyle\frac{a b}{c} \cdot y_{i} (a q, b q, c q) & (i = 3, 4). 
\end{cases}
\end{align*}
\end{lem}
By using a general theory of linear difference equations 
in \cite[Theorem~$2. 15$, p.$62$]{Elaydi}, we can verify the following lemma. 
\begin{lem}\label{lem:2}
On the assumption E1, 
$y_{i} (a, b; c; x) \, (i = 1, 2)$ are linearly independent solutions around $x = 0$ 
of $E_{q} (a, b, c)$, and $y_{i} (a, b; c; x) \, (i = 3, 4)$ are linearly independent solutions around $x = \infty$. 
\end{lem}

\subsection{Contiguity operators}\mbox{} \\

We define first order $q\mbox{-}$differential operators called contiguity operators, 
and by combining these operators, we obtain a $q\mbox{-}$differential operator $H (k, l, m)$ 
which send the parameters $a,\ b$ and $c$ to $a q^{k},\ b q^{l}$ and $c q^{m}$, respectively. 

Let $H_{j} (a, b, c)$ and $B_{j} (a, b, c)$ be the first order $q\mbox{-}$differential operators defined by 
\begin{align*}
H_{1} (a, b, c) &:= \frac{1 - a}{1 - q} + a x \Delta_{q}, \allowdisplaybreaks \\
H_{2} (a, b, c) &:= \frac{1 - b}{1 - q} + b x \Delta_{q}, \allowdisplaybreaks \\
H_{3} (a, b, c) &:= \frac{(1 - q) c}{(c - a) (c - b)} 
\left\{- a - b + a b + \frac{a b}{c} + (1 - q) (c - a b x) \Delta_{q} \right\}, \allowdisplaybreaks \\
B_{1} (a, b, c) &:= \frac{(1 - q) a}{(q - a) (c - a)} 
\left\{- q + \frac{c q}{a} + a (1 - b) x - (1 - q) x (c - a b x) \Delta_{q} \right\}, \allowdisplaybreaks \\
B_{2} (a, b, c) &:= \frac{(1 - q) b}{(q - b) (c - b)} 
\left\{- q + \frac{c q}{b} + b (1 - a) x - (1 - q) x (c - a b x) \Delta_{q} \right\}, \allowdisplaybreaks \\
B_{3} (a, b, c) &:= \frac{q - c}{q (1 - q)} + \frac{c}{q} x \Delta_{q}. 
\end{align*}
See \cite[Formulas $42.,\ 43.$ and $44.$, p.$292$]{Heine} or \cite[Exercises $1.9$ (i) and (ii), p.$22$]{GR} 
for the definitions of $H_{1},\ H_{2}$ and $B_{3}$, and 
see \cite[Remark~$2.1.4$, p.$46$]{IKSY} for the method for deriving operators 
$B_{1},\ B_{2}$ and $H_{3}$ from $H_{1},\ H_{2}$ and $B_{3}$, respectively. 

The operators $H_{j}$ ($j = 1, 2, 3$) 
increase the parameters $a,\ b$ and $c$ by $q$ times, respectively, 
while the operators $B_{j}$ ($j = 1, 2, 3$) decrease these parameters by $q$ times. 
Hence, we call these six operators contiguity operators. 
In fact, by performing direct calculations, we can prove the following lemma used in Section~$2. 4$ 
to obtain expressions for $\tilde{Q}$ and $\tilde{R}$. 
\begin{lem}\label{lem:3}
Let us denote $y_{i} (a, b; c; x),\ y_{i} (a q^{\pm 1}, b; c; x),\ y_{i} (a, b q^{\pm 1}; c; x)$ 
and $y_{i} (a, b; c q^{\pm 1}; x)$ by $y_{i},\ y_{i} (a q^{\pm 1}),\ y_{i} (b q^{\pm 1})$ 
and $y_{i} (c q^{\pm 1})$, respectively. 
Then, we have 
\begin{align*}
H_{1} (a, b, c) y_{i} &= y_{i} (a q), 
& 
B_{1} (a, b, c) y_{i} &= y_{i} (a q^{-1}), \allowdisplaybreaks \\
H_{2} (a, b, c) y_{i} &= y_{i} (b q), 
& 
B_{2} (a, b, c) y_{i} &= y_{i} (b q^{-1}), \allowdisplaybreaks \\
H_{3} (a, b, c) y_{i} &= y_{i} (c q), 
& 
B_{3} (a, b, c) y_{i} &= y_{i} (c q^{-1}) 
\end{align*}
for $i = 1$ and $2$. Also, we have 
\begin{align*}
H_{1} (a, b, c) y_{i} &= - a \cdot y_{i} (a q), 
& 
B_{1} (a, b, c) y_{i} &= - \frac{q}{a} \cdot y_{i} (a q^{-1}), \allowdisplaybreaks \\
H_{2} (a, b, c) y_{i} &= - b \cdot y_{i} (b q), 
& 
B_{2} (a, b, c) y_{i} &= - \frac{q}{b} \cdot y_{i} (b q^{-1}), \allowdisplaybreaks \\
H_{3} (a, b, c) y_{i} &= - \frac{1}{c} \cdot y_{i} (c q), 
& 
B_{3} (a, b, c) y_{i} &= - \frac{c}{q} \cdot y_{i} (c q^{-1}) 
\end{align*}
for $i = 3$ and $4$. 
\end{lem}
Let $S_{q} (a, b, c)$ denote the solution space of $E_{q} (a, b, c)$ 
on a simply connected domain in $\mathbb{C} - \left\{0 \right\}$. 
From Lemmas~$\ref{lem:2}$ and $\ref{lem:3}$, on the assumption E2, the mappings 
\begin{align*}
H_{1} (a, b, c) &: S_{q} (a, b, c) \to S_{q} (a q, b, c), \allowdisplaybreaks \\
H_{2} (a, b, c) &: S_{q} (a, b, c) \to S_{q} (a, b q, c), \allowdisplaybreaks \\
H_{3} (a, b, c) &: S_{q} (a, b, c) \to S_{q} (a, b, c q) 
\end{align*}
are linear isomorphisms, and their inverse mappings are given by 
$B_{1} (a q, b, c),\ B_{2} (a, b q, c)$ and $B_{3} (a, b, c q)$, respectively. 
Therefore, 
by combining $H_{j}$ and $B_{j}$ ($j = 1, 2, 3$), 
we obtain a linear isomorphism 
\begin{align*}
H (k, l, m) : S_{q} (a, b, c) \to S_{q} (a q^{k}, b q^{l}, c q^{m}) 
\end{align*}
for any integers $k,\ l$ and $m$.

\subsection{Expressions for $\tilde{Q}$ and $\tilde{R}$ as ratios of infinite series}\mbox{} \\

We express $\tilde{Q}$ and $\tilde{R}$ as ratios of infinite series 
defined by the sum of products of the series $y_{i}$ ($i = 1, 2, 3, 4$). 

Let $\mathbb{Q} (a, b, c, q, x)$ denote the field generated over $\mathbb{Q}$ by $a,\ b,\ c,\ q$ and $x$, 
and let us denote the ring of polynomials in $\Delta_{q}$ over $\mathbb{Q} (a, b, c, q, x)$ 
by $\mathbb{Q} (a, b, c, q, x) [\Delta_{q}]$. 
By the commutation relation $\Delta_{q} x = 1 + q x \Delta_{q}$, 
we can express $H (k, l, m)$ as 
\begin{align}\label{H}
H (k, l, m) = \rho (\Delta_{q}) \cdot L_{q} (a, b, c) + \hat{Q} \cdot \Delta_{q} + \hat{R}, 
\end{align}
where $\hat{Q},\ \hat{R} \in \mathbb{Q} (a, b, c, q, x)$ and 
$\rho (\Delta_{q}) \in \mathbb{Q} (a, b, c, q, x) [\Delta_{q}]$. 
From Lemmas~$\ref{lem:1}\mbox{--}\ref{lem:3}$ and the definition of $y_{1}$, 
operating $(\ref{H})$ to $y_{1} (a, b; c; x)$ yields 
\begin{align}\label{3tr''}
\f{a q^{k}}{b q^{l}}{c q^{m}}{x} = \hat{Q} \cdot \f{a q}{b q}{c q}{x} + \hat{R} \cdot \f{a}{b}{c}{x}. 
\end{align}
As stated in Section~$2. 1$, the pair $(Q, R)$ in $(\ref{3tr})$ are uniquely determined by $(k, l, m)$. 
It turns out that the pair $(\tilde{Q}, \tilde{R})$ in $(\ref{3tr'})$ are also uniquely determined by $(k, l, m)$. 
Therefore, by comparing $(\ref{3tr''})$ with $(\ref{3tr'})$, we find that $\hat{Q} = \tilde{Q}$ and $\hat{R} = \tilde{R}$. 
Namely, we have 
\begin{align}\label{H'}
H (k, l, m) = \rho (\Delta_{q}) \cdot L_{q} (a, b, c) + \tilde{Q} \cdot \Delta_{q} + \tilde{R}. 
\end{align}
From Lemmas~$\ref{lem:1}\mbox{--}\ref{lem:3}$, operating $(\ref{H'})$ to $y_{i} (a, b; c; x)$ (i = 1, 2, 3, 4) yields 
\begin{align}
\y{i}{a q^{k}}{b q^{l}}{c q^{m}}{x} 
&= \tilde{Q} \cdot \y{i}{a q}{b q}{c q}{x} + \tilde{R} \cdot \y{i}{a}{b}{c}{x} \quad (i = 1, 2), 
\label{3tr_y1,y2} \allowdisplaybreaks \\
\lambda_{0} \cdot \y{i}{a q^{k}}{b q^{l}}{c q^{m}}{x} 
&= - \frac{a b}{c} \tilde{Q} \cdot \y{i}{a q}{b q}{c q}{x} + \tilde{R} \cdot \y{i}{a}{b}{c}{x} \quad (i = 3, 4), 
\label{3tr_y3,y4}
\end{align}
where $\lambda_{0} := (-1)^{k + l - m} a^{k} b^{l} c^{-m} q^{\left\{k (k - 1) + l (l - 1) - m (m - 1) \right\} / 2}$. 
By solving two equations $(\ref{3tr_y1,y2})$ for $\tilde{Q}$ and $\tilde{R}$, we have 
\begin{align}
\tilde{Q} = \tilde{Q} (k, l, m) 
&= \frac{\displaystyle\Y{k}{l}{m}{a}{b}{c}{x}}{\displaystyle\Y{1}{1}{1}{a}{b}{c}{x}}, \label{tQ,Y} \allowdisplaybreaks \\
\tilde{R} = \tilde{R} (k, l, m) 
&= - \frac{\displaystyle\Y{k - 1}{l - 1}{m - 1}{a q}{b q}{c q}{x}}{\displaystyle\Y{1}{1}{1}{a}{b}{c}{x}}, \label{tR,Y}
\end{align}
where $Y$ is the infinite series defined by 
\begin{align}\label{def:Y}
\Y{k}{l}{m}{a}{b}{c}{x} := 
\y{1}{a q^{k}}{b q^{l}}{c q^{m}}{x} \y{2}{a}{b}{c}{x} 
- \y{2}{a q^{k}}{b q^{l}}{c q^{m}}{x} \y{1}{a}{b}{c}{x}. 
\end{align}
On the other hand, by solving two equations $(\ref{3tr_y3,y4})$ for $\tilde{Q}$ and $\tilde{R}$, we have 
\begin{align}
\tilde{Q} = \tilde{Q} (k, l, m) 
&= - \lambda_{0} \frac{c}{a b} 
\frac{\displaystyle\Z{k}{l}{m}{a}{b}{c}{x}}{\displaystyle\Z{1}{1}{1}{a}{b}{c}{x}}, \allowdisplaybreaks \label{tQ,Z} \\
\tilde{R} = \tilde{R} (k, l, m) 
&= - \lambda_{0} 
\frac{\displaystyle\Z{k - 1}{l - 1}{m - 1}{a q}{b q}{c q}{x}}{\displaystyle\Z{1}{1}{1}{a}{b}{c}{x}}, \label{tR,Z}
\end{align}
where $\tilde{Y}$ is the infinite series defined by 
\begin{align}\label{def:Z}
\Z{k}{l}{m}{a}{b}{c}{x} := 
\y{3}{a q^{k}}{b q^{l}}{c q^{m}}{x} \y{4}{a}{b}{c}{x} 
- \y{4}{a q^{k}}{b q^{l}}{c q^{m}}{x} \y{3}{a}{b}{c}{x}. 
\end{align}

In the next section, we prove Theorem~$\ref{main}$ from $(\ref{tQ,Y})$ and $(\ref{tR,Y})$, 
and we also prove Theorem~$\ref{main2}$ from $(\ref{tQ,Z})$ and $(\ref{tR,Z})$. 

\subsection{Proofs of Theorems~$\ref{main}$ and $\ref{main2}$}\mbox{} \\

We introduce four formulas for basic hypergeometric series. 
By using these formulas, 
we obtain expressions for $\tilde{Q}$ and $\tilde{R}$ as rational functions, 
and prove Theorems~$\ref{main}$ and $\ref{main2}$. 

Heine \cite[Formula XVIII., p.$325$]{Heine} showed that 
\begin{align}\tag{$\vec{\Phi 1}$}
\p{a}{b}{c}{x} = \frac{(a b x / c)_{\infty}}{(x)_{\infty}} \p{c / a}{c / b}{c}{\frac{a b}{c} x}. 
\end{align}
This is a $q\mbox{-}$analogue of Euler's transformation formula
\begin{align*}
\F{\alpha}{\beta}{\gamma}{x} = (1 - x)^{\gamma - \alpha - \beta} \F{\gamma - \alpha}{\gamma - \beta}{\gamma}{x}. 
\end{align*}

Assume that $m_{0}, m_{1}, \ldots, m_{p}$ are non-negative integers. 
Gasper showed that 
\begin{align}
&{}_{p + 2}\phi_{p + 1} \biggl(\genfrac..{0pt}{}{a_{0}, \,b_{0}, \,b_{1} q^{m_{1}}, \dotsc, b_{p} q^{m_{p}}}{b_{0} q^{m_{0} + 1}, \,b_{1}, \dotsc, b_{p}}\,; q,\, a_{0}^{-1} q^{m_{0} + 1 - (m_{1} + \dotsm + m_{p})}\biggr) \tag{$\vec{\Phi 2}$} \\
&= \frac{(q)_{\infty} (b_{0} q / a_{0})_{\infty} (b_{0} q)_{m_{0}} (b_{1} / b_{0})_{m_{1}} \dotsm (b_{p} / b_{0})_{m_{p}}}{(b_{0} q)_{\infty} (q / a_{0})_{\infty} (q)_{m_{0}} (b_{1})_{m_{1}} \dotsm (b_{p})_{m_{p}}} 
b_{0}^{m_{1} + \dotsm + m_{p} - m_{0}} \notag \\
&\times {}_{p + 2}\phi_{p + 1}\biggl(\genfrac..{0pt}{}{q^{-m_{0}}, \,b_{0}, \,b_{0} q / b_{1}, \dotsc, b_{0} q / b_{p}}{b_{0} q / a_{0}, \,b_{0} q^{1 - m_{1}} / b_{1}, \dotsc, b_{0} q^{1 - m_{p}} / b_{p}}\,; q,\, q\biggr), 
\quad \lvert a_{0}^{-1} q^{m_{0} + 1 - (m_{1} + \dotsm + m_{p})} \rvert < 1, \notag \allowdisplaybreaks \\
&{}_{p + 1}\phi_{p} \biggl(\genfrac..{0pt}{}{a_{0}, \,b_{1} q^{m_{1}}, \dotsc, b_{p} q^{m_{p}}}{b_{1}, \dotsc, b_{p}}\,; q,\, a_{0}^{-1} q^{- (m_{1} + \dotsm + m_{p})}\biggr) = 0, \quad \lvert a_{0}^{-1} q^{- (m_{1} + \dotsm + m_{p})} \rvert < 1, \tag{$\vec{\Phi 3}$} \allowdisplaybreaks \\
&{}_{p + 1}\phi_{p} \biggl(\genfrac..{0pt}{}{q^{-n}, \,b_{1} q^{m_{1}}, \dotsc, b_{p} q^{m_{p}}}{b_{1}, \dotsc, b_{p}}\,; q,\, q \biggr) = 0, \quad n > m_{1} + \dotsm + m_{p}, \tag{$\vec{\Phi 4}$}
\end{align}
which are $(20),\ (8)$ and $(15)$ in \cite{Gasper}, respectively. 
The formula $(\vec{\Phi 2})$ is a $q\mbox{-}$analogue of 
\begin{align*}
&{}_{p + 2}F_{p + 1} \biggl(\genfrac..{0pt}{}{\alpha, \,\beta, \,\beta_{1} + m_{1}, \dotsc, \beta_{p} + m_{p}}{\beta + \gamma + 1, \,\beta_{1}, \dotsc, \beta_{p}}\,; 1\biggr) \\
&= \frac{\G (\beta + \gamma + 1) \G (1 - \alpha)}{\G (\beta + 1 - \alpha) \G (\gamma + 1)} 
\frac{(\beta_{1} - \beta, m_{1}) \dotsm (\beta_{p} - \beta, m_{p})}{(\beta_{1}, m_{1}) \dotsm (\beta_{p}, m_{p})} \\
&\quad \times {}_{p + 2}F_{p + 1} \biggl(\genfrac..{0pt}{}{-\gamma, \,\beta, \,1 + \beta - \beta_{1}, \dotsc, 1 + \beta - \beta_{p}}{\beta + 1 - \alpha, \,1 + \beta - \beta_{1} - m_{1}, \dotsc, 1 + \beta - \beta_{p} - m_{p}}\,; 1\biggr), 
\end{align*}
provided $\Re (\gamma - \alpha) > m_{1} + \dotsm + m_{p} - 1$, where $(\alpha, n) := \G (\alpha + n) / \G (\alpha)$. 
This formula was also given by Gasper \cite[$(18)$]{Gasper}. 
Both $(\vec{\Phi 3})$ and $(\vec{\Phi 4})$ are $q\mbox{-}$analogues of Karlsson's \cite[$(12)$]{Karlsson} summation formula 
\begin{align*}
{}_{p + 1}F_{p} \biggl(\genfrac..{0pt}{}{\alpha, \,\beta_{1} + m_{1}, \dotsc, \beta_{p} + m_{p}}{\beta_{1}, \dotsc, \beta_{p}}\,; 1\biggr) = 0, 
\quad \Re (- \alpha) > m_{1} + \dotsm + m_{p}. 
\end{align*}

First, we prove Theorem~$\ref{main}$. 
The following is the definition of $P$ in Theorem~$\ref{main}$. 
\begin{definition}\label{def:P,A,B}
For any integers $k,\ l,\ m$ with $k \leq l$, let $P$ be the polynomial in $x$ defined by 
\begin{align*}
P = P \c{k}{l}{m}{a}{b}{c}{x} := 
\begin{cases}
\displaystyle\sum_{n = 0}^{l - 1} \left(A_{n} - B_{n - m} \right) x^{n} 
& (m \geq 0,\, k + l - m \geq 0), \\
\displaystyle\sum_{n = 0}^{m - k - 1} \left(\tilde{A}_{n} - \tilde{B}_{n - m} \right) x^{n} 
& (m \geq 0,\, k + l - m < 0), \\
\displaystyle\sum_{n = 0}^{l - m - 1} \left(A_{n + m} - B_{n} \right) x^{n} 
& (m < 0,\, k + l - m \geq 0), \\
\displaystyle\sum_{n = 0}^{- k - 1} \left(\tilde{A}_{n + m} - \tilde{B}_{n} \right) x^{n} 
& (m < 0,\, k + l - m < 0), 
\end{cases}
\end{align*}
where $A_{n} = B_{n} = \tilde{A}_{n} = \tilde{B}_{n} := 0$ for any negative integer $n$, and 
\begin{align*}
A_{n} 
&:= \frac{- c (a q / c)_{k - m} (b q / c)_{l - m}}{(q^{2} / c)_{- m - 1}} 
\frac{(c)_{m - n - 1}}{(q^{-n})_{n} (a)_{k - n} (b)_{l - n}} q^{m - n - 1}
\P{q^{-n}}{c q^{m - n - 1}}{a}{b}{c}{a q^{k - n}}{b q^{l - n}}{q}, \allowdisplaybreaks \\
B_{n} 
&:= \frac{(a q / c)_{n} (b q / c)_{n}}{(q)_{n} (q^{2} / c)_{n}} 
\P{q^{-n}}{c q^{- n - 1}}{c q^{m - k} / a}{c q^{m - l} / b}{c q^{m}}{c q^{- n} / a}{c q^{- n} / b}{q^{k + l - m + 1}}, \allowdisplaybreaks \\
\tilde{A}_{n} 
&:= \frac{(c)_{m}}{(a)_{k} (b)_{l}} 
\frac{(a q / c)_{n + k - m} (b q / c)_{n + l - m}}{(q)_{n} (q^{2} / c)_{n - m}} 
\P{q^{-n}}{c q^{m - n - 1}}{c / a}{c / b}{c}{c q^{m - k - n} / a}{c q^{m - l - n} / b}{q^{m - k - l + 1}}, 
\allowdisplaybreaks \\
\tilde{B}_{n} 
&:= \frac{(a q^{-n})_{n} (b q^{-n})_{n}}{(q^{-n})_{n} (c q^{-n - 1})_{n}} q^{-n} 
\P{q^{-n}}{c q^{- n - 1}}{a q^{k}}{b q^{l}}{c q^{m}}{a q^{-n}}{b q^{-n}}{q} 
\end{align*}
for any non-negative integer $n$. 
\end{definition}
To prove Theorem~$\ref{main}$, we give some lemmas. 
The following lemma follows from $(\vec{\Phi 1})$. 
\begin{lem}\label{Y,A,B}
The series $Y$ can be written as 
\begin{subequations}
\begin{align}
&\Y{k}{l}{m}{a}{b}{c}{x} \nonumber \\
&\quad = - \lambda_{1} (1 - q)^{m - k - l} 
x^{1 - \gamma - m} \frac{(a)_{k} (b)_{l}}{(c)_{m}} \frac{(a b q^{k + l - m} x / c)_{\infty}}{(x)_{\infty}} 
\left\{\sum_{n = 0}^{\infty} A_{n} x^{n} - x^{m} \sum_{n = 0}^{\infty} B_{n} x^{n} \right\}, \label{Y1} \allowdisplaybreaks \\
&\quad = - \lambda_{1} (1 - q)^{m - k - l} 
x^{1 - \gamma - m} \frac{(a)_{k} (b)_{l}}{(c)_{m}} \frac{(a b x / c)_{\infty}}{(x)_{\infty}} 
\left\{\sum_{n = 0}^{\infty} \tilde{A}_{n} x^{n} - x^{m} \sum_{n = 0}^{\infty} \tilde{B}_{n} x^{n} \right\}, \label{Y2} 
\end{align}
\end{subequations}
where 
\begin{align*}
\lambda_{1} := 
(1 - q)^{2 (\gamma - \alpha - \beta + 1)} 
\frac{(q)_{\infty}^{2} (c)_{\infty} (q^{2}/ c)_{\infty}}{(a)_{\infty} (b)_{\infty} (a q / c)_{\infty} (b q / c)_{\infty}}. 
\end{align*}
\end{lem}
\begin{proof}
Applying $(\vec{\Phi 1})$ to $y_{1} (a q^{k}, b q^{l}; c q^{m}; x)$ and $y_{2} (a q^{k}, b q^{l}; c q^{m}; x)$ 
in $(\ref{def:Y})$ yields 
\begin{align*}
&\Y{k}{l}{m}{a}{b}{c}{x} \\
&\quad = (1 - q)^{2 (\gamma - \alpha - \beta + 1) + m - k - l} 
\frac{(q)_{\infty}^{2} (c q^{m})_{\infty} (q^{2}/ c)_{\infty}}{(a q^{k})_{\infty} (b q^{l})_{\infty} (a q / c)_{\infty} (b q / c)_{\infty}} x^{1 - \gamma - m} \frac{(a b q^{k + l - m} x / c)_{\infty}}{(x)_{\infty}} \\
&\quad \quad \times \left\{x^{m} \p{c q^{m - k} / a}{c q^{m - l} / b}{c q^{m}}{\frac{a b q^{k + l - m}}{c} x} 
\p{a q / c}{b q / c}{q^{2} / c}{x} \right. \\
&\quad \quad \quad \quad \left. - \frac{(a q / c)_{k - m} (b q / c)_{l - m} (c)_{m}}{(q^{2} / c)_{-m} (a)_{k} (b)_{l}} 
\p{q^{1 - k} / a}{q^{1 - l} / b}{q^{2 - m} / c}{\frac{a b q^{k + l - m}}{c} x} \p{a}{b}{c}{x} \right\}. 
\end{align*}
From 
$(a)_{n - i} = (-a)^{-i} q^{(1 - n) i + \binom{i}{2}} (a)_{n} / (a^{-1} q^{1 - n})_{i}$, 
we can find that the coefficient of $x^{n}$ in 
\begin{align*}
\p{c q^{m - k} / a}{c q^{m - l} / b}{c q^{m}}{\frac{a b q^{k + l - m}}{c} x} 
\p{a q / c}{b q / c}{q^{2} / c}{x} 
\end{align*}
is equal to $B_{n}$, and the coefficient of $x^{n}$ in 
\begin{align*}
\frac{(a q / c)_{k - m} (b q / c)_{l - m} (c)_{m}}{(q^{2} / c)_{-m} (a)_{k} (b)_{l}} 
\p{q^{1 - k} / a}{q^{1 - l} / b}{q^{2 - m} / c}{\frac{a b q^{k + l - m}}{c} x} \p{a}{b}{c}{x} 
\end{align*}
is equal to $A_{n}$. 
This proves $(\ref{Y1})$. 

Applying $(\vec{\Phi 1})$ to $y_{1} (a, b; c; x)$ and $y_{2} (a, b; c; x)$ 
in $(\ref{def:Y})$ yields 
\begin{align*}
&\Y{k}{l}{m}{a}{b}{c}{x} \\
&\quad = (1 - q)^{2 (\gamma - \alpha - \beta + 1) + m - k - l} 
\frac{(q)_{\infty}^{2} (c q^{m})_{\infty} (q^{2}/ c)_{\infty}}{(a q^{k})_{\infty} (b q^{l})_{\infty} (a q / c)_{\infty} (b q / c)_{\infty}} x^{1 - \gamma - m} \frac{(a b x / c)_{\infty}}{(x)_{\infty}} \\
&\quad \quad \times \left\{x^{m} \p{a q^{k}}{b q^{l}}{c q^{m}}{x} \p{q / a}{q / b}{q^{2} / c}{\frac{a b}{c} x} \right. \\
&\quad \quad \quad \quad \left. - \frac{(a q / c)_{k - m} (b q / c)_{l - m} (c)_{m}}{(q^{2} / c)_{-m} (a)_{k} (b)_{l}} 
\p{a q^{k + 1 - m} / c}{b q^{l + 1 - m} / c}{q^{2 - m} / c}{x} \p{c / a}{c / b}{c}{\frac{a b}{c} x} \right\}. 
\end{align*}
From 
$(a)_{n - i} = (- a)^{-i} q^{(1 - n) i + \binom{i}{2}} (a)_{n} / (a^{-1} q^{1 - n})_{i}$, 
we can find that the coefficient of $x^{n}$ in 
\begin{align*}
\p{a q^{k}}{b q^{l}}{c q^{m}}{x} \p{q / a}{q / b}{q^{2} / c}{\frac{a b}{c} x} 
\end{align*}
is equal to $\tilde{B}_{n}$, and the coefficient of $x^{n}$ in 
\begin{align*}
\frac{(a q / c)_{k - m} (b q / c)_{l - m} (c)_{m}}{(q^{2} / c)_{-m} (a)_{k} (b)_{l}} 
\p{a q^{k + 1 - m} / c}{b q^{l + 1 - m} / c}{q^{2 - m} / c}{x} \p{c / a}{c / b}{c}{\frac{a b}{c} x} 
\end{align*}
is equal to $\tilde{A}_{n}$. 
This proves $(\ref{Y2})$. 
\end{proof}

From $(\vec{\Phi 2})\mbox{--}(\vec{\Phi 4})$, we obtain the following lemma. 
\begin{lem}\label{(i)-(iv)}
For any integers $k,\ l$ and $m$ with $k \leq l$, we have 
\begin{align}
\tag{i} m \geq 0, \; k + l - m \geq 0 \; &\Rightarrow \; A_{n} - B_{n - m} = 0 \quad (n \geq l). \allowdisplaybreaks \\
\tag{ii} m \geq 0, \; k + l - m < 0 \; &\Rightarrow \; \tilde{A}_{n} - \tilde{B}_{n - m} = 0 \quad (n \geq m - k). 
\allowdisplaybreaks \\
\tag{iii} m < 0, \; k + l - m \geq 0 \; &\Rightarrow \; A_{n + m} - B_{n} = 0 \quad (n \geq l - m). \allowdisplaybreaks \\
\tag{iv} m < 0, \; k + l - m < 0 \; &\Rightarrow \; \tilde{A}_{n + m} - \tilde{B}_{n} = 0 \quad (n \geq - k). 
\end{align}
\end{lem}
\begin{proof}
Let us prove (i). 
First, by using $(\vec{\Phi 2})$, we show that 
\begin{align}\tag{i-1}
m \geq 0, \; k + l - m \geq 0 \; \Rightarrow \; A_{n} - B_{n - m} = 0 \quad (n \geq \max\left\{l, m \right\}). 
\end{align}
By substituting $p = 2,\ a_{0} = q^{m - n},\ b_{0} = c q^{m - n - 1}$, 
$b_{1} = c q^{m - n} / a,\ b_{2} = c q^{m - n} / b$, 
$m_{0} = n,\ m_{1} = n - k$ and $m_{2} = n - l$ into $(\vec{\Phi 2})$, 
we find that if $k + l - m \geq 0$ and $n \geq \max \left\{0, l \right\}$, then 
\begin{align*}
&\P{q^{m - n}}{c q^{m - n - 1}}{c q^{m - k}/ a}{c q^{m - l} / b}{c q^{m}}{c q^{m - n} / a}{c q^{m - n} / b}{q^{k + l - m + 1}} \\
&= \frac{(q)_{\infty} (c)_{\infty} (c q^{m - n})_{n} (q / a)_{n - k} (q / b)_{n - l}}{(c q^{m - n})_{\infty} (q^{n + 1 - m})_{\infty} (q)_{n} (c q^{m - n}/ a)_{n - k} (c q^{m - n} / b)_{n - l}} c^{n - k - l} q^{(m - n - 1) (n - k - l)} \\
&\quad \times \P{q^{- n}}{c q^{m - n - 1}}{a}{b}{c}{a q^{k - n}}{b q^{l - n}}{q} \allowdisplaybreaks \\
&= \frac{(c)_{m} (q / a)_{n - k} (q / b)_{n - l}}{(q^{n + 1 - m})_{m} (c q^{m - n} / a)_{n - k} (c q^{m - n} / b)_{n - l}} c^{n - k - l} q^{(m - n - 1) (n - k - l)} 
\P{q^{- n}}{c q^{m - n - 1}}{a}{b}{c}{a q^{k - n}}{b q^{l - n}}{q}. 
\end{align*}
Thus, it turns out that if $k + l - m \geq 0$ and $n \geq \max \left\{0, l, m \right\}$, then 
\begin{align*}
B_{n - m} 
&= \frac{(a q / c)_{n - m} (b q / c)_{n - m}}{(q)_{n - m} (q^{2} / c)_{n - m}} \\
&\ \times 
\frac{(c)_{m} (q / a)_{n - k} (q / b)_{n - l} c^{n - k - l}}{(q^{n + 1 - m})_{m} (c q^{m - n} / a)_{n - k} (c q^{m - n} / b)_{n - l}} 
q^{(m - n - 1) (n - k - l)} 
\P{q^{- n}}{c q^{m - n - 1}}{a}{b}{c}{a q^{k - n}}{b q^{l - n}}{q} \allowdisplaybreaks \\
&= \frac{- c (a q / c)_{k - m} (b q / c)_{l - m}}{(q^{2} / c)_{- m - 1}} 
\frac{(c)_{m - n - 1}}{(q^{-n})_{n} (a)_{k - n} (b)_{l - n}} q^{m - n - 1}
\P{q^{-n}}{c q^{m - n - 1}}{a}{b}{c}{a q^{k - n}}{b q^{l - n}}{q} \allowdisplaybreaks \\
&= A_{n}. 
\end{align*}
This proves (i-1). 
Next, by using $(\vec{\Phi 4})$, we show that 
\begin{align}\tag{i-2}
m \geq 0, \; k + l - m \geq 0, \; l < m \; \Rightarrow \; A_{n} - B_{n - m} = 0 \quad (l \leq n \leq m - 1). 
\end{align}
By substituting $p = 3,\ b_{1} = c,\ b_{2} = a q^{k - n},\ b_{3} = b q^{l - n}$, 
$m_{1} = m - n - 1,\ m_{2} = n - k$ and $m_{3} = n - l$ into $(\vec{\Phi 4})$, 
we find that if $k + l - m \geq 0,\ l < m$ and $l \leq n \leq m - 1$, then 
\begin{align*}
\P{q^{-n}}{c q^{m - n - 1}}{a}{b}{c}{a q^{k - n}}{b q^{l - n}}{q} = 0. 
\end{align*}
Thus, it turns out that if $k + l - m \geq 0,\ l < m$ and $l \leq n \leq m - 1$, then $A_{n} - B_{n - m} = A_{n} = 0$. 
This proves (i-2). 
Consequently, from (i-1) and (i-2), we have (i). 

In the same way as we proved (i-1), by using $(\vec{\Phi 2})$, we can show that 
\begin{align}
\tag{ii-1} &m \geq 0, \; k + l - m < 0 \; \Rightarrow \; \tilde{A}_{n} - \tilde{B}_{n - m} = 0 
\quad (n \geq \max \left\{m, m - k \right\}). \allowdisplaybreaks \\
\tag{iii-1} &m < 0, \; k + l - m \geq 0 \; \Rightarrow \; A_{n + m} - B_{n} = 0 
\quad (n \geq \max \left\{-m, l - m \right\}). \allowdisplaybreaks \\
\tag{iv-1} &m < 0, \; k + l - m < 0 \; \Rightarrow \; \tilde{A}_{n + m} - \tilde{B}_{n} = 0 
\quad (n \geq \max \left\{-m, -k \right\}). 
\end{align}
Also, in the same way as we proved (i-2), by using $(\vec{\Phi 3})$ or $(\vec{\Phi 4})$, we can show that 
\begin{align}
\tag{ii-2} &m \geq 0, \; k + l - m < 0, \; m - k < m \; \Rightarrow \; \tilde{A}_{n} - \tilde{B}_{n - m} = 0 
\quad (m - k \leq n \leq m - 1). \allowdisplaybreaks \\
\tag{iii-2} &m < 0, \; k + l - m \geq 0, \; l - m < -m \; \Rightarrow \; A_{n + m} - B_{n} = 0 
\quad (l - m \leq n \leq - m - 1). \allowdisplaybreaks \\
\tag{iv-2} &m < 0, \; k + l - m < 0, \; - k < - m \; \Rightarrow \; \tilde{A}_{n + m} - \tilde{B}_{n} = 0 
\quad (- k \leq n \leq - m - 1). 
\end{align}
From these results, we can complete the proof of Lemma~$\ref{(i)-(iv)}$. 
\end{proof}

It immediately follows from Lemmas~$\ref{Y,A,B}$ and $\ref{(i)-(iv)}$ that 
\begin{lem}\label{lem:Y}
For any integers $k,\ l$ and $m$ with $k \leq l$, the series $Y$ can be expressed as 
\begin{align}
\Y{k}{l}{m}{a}{b}{c}{x} 
&= - \lambda_{1} (1 - q)^{m - k - l} \frac{(a)_{k} (b)_{l}}{(c)_{m}} \label{Y} \\
&\quad \times x^{1 - \gamma - \max \left\{m, 0 \right\}} 
\frac{(a b q^{\max \left\{k + l - m, 0 \right\}} x / c)_{\infty}}{(x)_{\infty}} P \c{k}{l}{m}{a}{b}{c}{x}, \nonumber 
\end{align}
where $\lambda_{1}$ is defined as in Lemma~$\ref{Y,A,B}$. 
\end{lem}
From $(\ref{tQ,Y}),\ (\ref{tR,Y})$ and Lemma~$\ref{lem:Y}$, 
we obtain the following lemma which gives expressions for $\tilde{Q}$ and $\tilde{R}$ as rational functions. 
\begin{lem}\label{tQ,tR}
Assume that $k \leq l$. Then, the coefficients of $(\ref{3tr'})$ can be expressed as 
\begin{align*}
\tilde{Q} (k, l, m) 
&= - (1 - q)^{m - k - l + 1} \frac{c (a)_{k} (b)_{l}}{(q - c) (c)_{m}} 
\frac{x^{1 - \max \left\{m, 0 \right\}}}{(a b q x / c)_{\max \left\{k + l - m, 0 \right\} - 1}} 
P \c{k}{l}{m}{a}{b}{c}{x}, \allowdisplaybreaks \\
\tilde{R} (k, l, m) 
&= - (1 - q)^{m - k - l} \frac{(a)_{k} (b)_{l}}{(c)_{m}} 
\frac{x^{- \max \left\{m - 1, 0 \right\}}}{(a b q x / c)_{\max \left\{k + l - m - 1, 0 \right\}}} 
P \c{k - 1}{l - 1}{m - 1}{a q}{b q}{c q}{x}. 
\end{align*}
\end{lem}
\begin{proof}
By substituting $(k, l, m) = (1, 1, 1)$ into $(\ref{Y})$, we have
\begin{align}\label{Y(111)}
\Y{1}{1}{1}{a}{b}{c}{x} 
= \lambda_{1} \frac{x^{-\gamma}}{1 - q} \frac{(a b q x / c)_{\infty}}{(x)_{\infty}} \frac{q - c}{c}. 
\end{align}
From $(\ref{tQ,Y}),\ (\ref{Y})$ and $(\ref{Y(111)})$, we have the expression for $\tilde{Q}$. 

Replacing $k,\ l,\ m,\ a,\ b,\ c$ in $(\ref{Y})$ with respective 
$k - 1,\ l - 1,\ m - 1,\ a q,\ b q,\ c q$ yields 
\begin{align}
&\Y{k - 1}{l - 1}{m - 1}{a q}{b q}{c q}{x} \label{Y'} \\
&\quad = \lambda_{1} (1 - q)^{m - k - l - 1} \frac{q - c}{c} \frac{(a)_{k} (b)_{l}}{(c)_{m}} \nonumber \\
&\quad \quad \times x^{- \gamma - \max \left\{m - 1, 0 \right\}} 
\frac{(a b q^{\max \left\{k + l - m - 1, 0 \right\} + 1} x / c)_{\infty}}{(x)_{\infty}} 
P \c{k - 1}{l - 1}{m - 1}{a q}{b q}{c q}{x}. \nonumber 
\end{align}
From $(\ref{tR,Y}),\ (\ref{Y(111)})$ and $(\ref{Y'})$, we have the expression for $\tilde{R}$. 
Thus the lemma is proved. 
\end{proof}
From $(\ref{Q,tQ}),\ (\ref{R,tR})$ and Lemma~$\ref{tQ,tR}$, 
we can complete the proof of Theorem~$\ref{main}$.

Next, we prove Theorem~$\ref{main2}$. 
The following is the definition of $\tilde{P}$ in Theorem~$\ref{main2}$. 
\begin{definition}\label{def:til(P),C,D}
For any integers $k,\ l,\ m$ with $k \leq l$, let $\tilde{P}$ be the polynomial in $x$ defined by 
\begin{align*}
\tilde{P} = 
\tilde{P} \c{k}{l}{m}{a}{b}{c}{x} := 
\begin{cases}
\displaystyle\sum_{n = 0}^{l - 1} \left(C_{n} - D_{n + k - l} \right) x^{l - 1 - n} 
& (m \geq 0,\, k + l - m \geq 0), \\
\displaystyle\sum_{n = 0}^{m - k - 1} \left(\tilde{C}_{n} - \tilde{D}_{n + k - l} \right) x^{m - k - 1 - n} 
& (m \geq 0,\, k + l - m < 0), \\
\displaystyle\sum_{n = 0}^{l - m - 1} \left(C_{n} - D_{n + k - l} \right) x^{l - m - 1 - n} 
& (m < 0,\, k + l - m \geq 0), \\
\displaystyle\sum_{n = 0}^{- k - 1} \left(\tilde{C}_{n} - \tilde{D}_{n + k - l} \right) x^{- k - 1 - n} 
& (m < 0,\, k + l - m < 0), 
\end{cases}
\end{align*}
where $D_{n} = \tilde{D}_{n} := 0$ for any negative integer $n$, and 
\begin{align*}
C_{n} 
&:= \mu_{1} 
\frac{(b)_{n} (b q / c)_{n}}{(q)_{n} (b q / a)_{n}} \left(\frac{c q}{a b} \right)^{n} 
\P{q^{-n}}{a q^{-n} / b}{q^{1 - l} / b}{c q^{m - l} / b}{a q^{k - l + 1} / b}{q^{1 - n} / b}{c q^{-n} / b}{q}, 
\allowdisplaybreaks \\
D_{n} 
&:= \mu_{2} 
\frac{(q^{1 - k} / a)_{n} (c q^{m - k} / a)_{n}}{(q)_{n} (b q^{l - k + 1} / a)_{n}} q^{n} 
\P{q^{-n}}{a q^{k - l - n} / b}{a}{a q / c}{a q / b}{a q^{k - n}}{a q^{k - m - n + 1} / c}{q^{k + l - m + 1}}, 
\allowdisplaybreaks \\
\tilde{C}_{n} 
&:= \mu_{1} 
\frac{(a q^{k})_{n} (a q^{k - m + 1} / c)_{n}}{(q)_{n} (a q^{k - l + 1} / b)_{n}} 
\left(\frac{c q^{m - k - l + 1}}{a b} \right)^{n} 
\P{q^{-n}}{b q^{l - k - n} / a}{q / a}{c / a}{b q / a}{q^{1 - k - n} / a}{c q^{m - k - n} / a}{q}, 
\allowdisplaybreaks \\
\tilde{D}_{n} 
&:= \mu_{2} 
\frac{(q / b)_{n} (c / b)_{n}}{(q)_{n} (a q / b)_{n}} q^{n} 
\P{q^{-n}}{b q^{-n} / a}{b q^{l}}{b q^{l - m + 1} / c}{b q^{l - k + 1} / a}{b q^{-n}}{b q^{1 - n} / c}{q^{m - k - l + 1}} 
\end{align*}
with 
\begin{align*}
\mu_{1} &:= a^{m - k - l} \left(\frac{c}{a b} \right)^{k} q^{k (m - k - l + 1)} 
\frac{(a)_{k} (a q / c)_{k - m}}{(a q / b)_{k - l}}, \allowdisplaybreaks \\
\mu_{2} &:= b^{m - k - l} \left(\frac{c}{a b} \right)^{l} q^{l (m - k - l + 1)} 
\frac{(b)_{l} (b q / c)_{l - m}}{(b q / a)_{l - k}} 
\end{align*}
for any non-negative integer $n$. 
\end{definition}
To prove Theorem~$\ref{main2}$, we give some lemmas. 
In the same way as we proved Lemma~$\ref{Y,A,B}$, by using $(\vec{\Phi 1})$, we can show that 
\begin{lem}\label{Z,C,D}
The series $\tilde{Y}$ can be written as 
\begin{align*}
\Z{k}{l}{m}{a}{b}{c}{x} 
&= \lambda_{2} (1 - q)^{m - k - l} x^{- \alpha - \beta - k} 
\frac{(q / x)_{\infty}}{(c q^{m - k - l + 1} / (a b x))_{\infty}} 
\!\left\{\!\sum_{n = 0}^{\infty} C_{n} x^{-n} \!- x^{k - l} \sum_{n = 0}^{\infty} D_{n} x^{-n} \!\right\}\!, \allowdisplaybreaks \\
&= \lambda_{2} (1 - q)^{m - k - l} x^{- \alpha - \beta - k} 
\frac{(q / x)_{\infty}}{(c q / (a b x))_{\infty}} 
\left\{\sum_{n = 0}^{\infty} \tilde{C}_{n} x^{-n} - x^{k - l} \sum_{n = 0}^{\infty} \tilde{D}_{n} x^{-n} \right\}, 
\end{align*}
where 
\begin{align*}
\lambda_{2} := 
(1 - q)^{2 (\gamma - \alpha - \beta + 1)} (a b)^{\gamma - \alpha - \beta + 1} 
\frac{(q)_{\infty}^{2} (a q / b)_{\infty} (b q / a)_{\infty}}{(a)_{\infty} (b)_{\infty} (a q / c)_{\infty}(b q / c)_{\infty}}. 
\end{align*}
\end{lem}
Also, in the same way as we proved Lemma~$\ref{(i)-(iv)}$, by using $(\vec{\Phi 2})\mbox{--}(\vec{\Phi 4})$, 
we can show that 
\begin{lem}\label{(i')-(iv')}
For any integers $k,\ l$ and $m$ with $k \leq l$, we have 
\begin{align}
\tag{i} m \geq 0, \; k + l - m \geq 0 \; &\Rightarrow \; C_{n} - D_{n + k - l} = 0 \quad (n \geq l). \allowdisplaybreaks \\
\tag{ii} m \geq 0, \; k + l - m < 0 \; &\Rightarrow \; \tilde{C}_{n} - \tilde{D}_{n + k - l} = 0 \quad (n \geq m - k). 
\allowdisplaybreaks \\
\tag{iii} m < 0, \; k + l - m \geq 0 \; &\Rightarrow \; C_{n} - D_{n + k - l} = 0 \quad (n \geq l - m). 
\allowdisplaybreaks \\
\tag{iv} m < 0, \; k + l - m < 0 \; &\Rightarrow \; \tilde{C}_{n} - \tilde{D}_{n + k - l} = 0 \quad (n \geq - k). 
\end{align}
\end{lem}
It immediately follows from Lemmas~$\ref{Z,C,D}$ and $\ref{(i')-(iv')}$ that 
\begin{lem}\label{lem:Z}
For any integers $k,\ l$ and $m$ with $k \leq l$, the series $\tilde{Y}$ can be expressed as 
\begin{align}
\Z{k}{l}{m}{a}{b}{c}{x} 
&= \lambda_{2} (1 - q)^{m - k - l} x^{- \alpha - \beta + 1 - \max \left\{m, 0 \right\} - \max \left\{k + l - m, 0 \right\}} 
\label{Z} \\
&\quad \times \frac{(q / x)_{\infty}}{(c q^{1 - \max \left\{k + l - m, 0 \right\}} / (a b x))_{\infty}} 
\tilde{P} \c{k}{l}{m}{a}{b}{c}{x}, \nonumber
\end{align}
where $\lambda_{2}$ is defined as in Lemma~$\ref{Z,C,D}$. 
\end{lem}
From $(\ref{tQ,Z}),\ (\ref{tR,Z})$ and Lemma~$\ref{lem:Z}$, 
we obtain the following lemma which gives expressions for $\tilde{Q}$ and $\tilde{R}$ 
as rational functions. 
\begin{lem}\label{tQ,tR-2}
Assume that $k \leq l$. Then, the coefficients of $(\ref{3tr'})$ can be expressed as 
\begin{align*}
\tilde{Q} (k, l, m) 
&= (-1)^{k + l - m + M} (1 - q)^{m - k - l + 1} q^{\left\{k (k - 1) + l (l - 1) - m (m - 1) + M (M - 1) \right\} / 2} \\
&\quad \times \frac{a^{k} b^{l}}{(b - a) c^{m - 1}} \left(\frac{a b}{c} \right)^{M} 
\frac{x^{1 - \max \left\{m, 0 \right\}}}{(a b q x / c)_{M - 1}} 
\tilde{P} \c{k}{l}{m}{a}{b}{c}{x}, \allowdisplaybreaks \\
\tilde{R} (k, l, m) 
&= (-1)^{k + l - m - 1+ M'} (1 - q)^{m - k - l} q^{\left\{k (k - 1) + l (l - 1) - m (m - 1) + M' (M' + 1) \right\} / 2} \\
&\quad \times \frac{(1 - a) (1 - b) a^{k} b^{l}}{(b - a) c^{m}} \left(\frac{a b}{c} \right)^{M' - 1} 
\frac{x^{- \max \left\{m - 1, 0 \right\}}}{(a b q x / c)_{M'}} 
\tilde{P} \c{k - 1}{l - 1}{m - 1}{a q}{b q}{c q}{x}, 
\end{align*}
where $M := \max\left\{k + l - m, 0 \right\}$ and $M' := \max\left\{k + l - m - 1, 0 \right\}$. 
\end{lem}
\begin{proof}
By substituting $(k, l, m) = (1, 1, 1)$ into $(\ref{Z})$, we have
\begin{align}\label{Z(111)}
\Z{1}{1}{1}{a}{b}{c}{x} 
= \lambda_{2} \frac{x^{- \alpha - \beta - 1}}{1 - q} 
\frac{(q / x)_{\infty}}{(c / (a b x))_{\infty}} \frac{(b - a) c}{a^{2} b^{2}}. 
\end{align}
From $(\ref{tQ,Z}),\ (\ref{Z})$ and $(\ref{Z(111)})$, we have the expression for $\tilde{Q}$. 

Replacing $k,\ l,\ m,\ a,\ b,\ c$ in $(\ref{Z})$ with respective 
$k - 1,\ l - 1,\ m - 1,\ a q,\ b q,\ c q$ yields 
\begin{align}
&\Z{k - 1}{l - 1}{m - 1}{a q}{b q}{c q}{x} \label{Z'} \\
&\quad = \lambda_{2} (1 - q)^{m - k - l - 1} 
x^{- \alpha - \beta - 1 - \max \left\{m - 1, 0 \right\}- \max \left\{k + l - m - 1, 0 \right\}} \nonumber \\
&\quad \quad \times 
\frac{(q / x)_{\infty}}{(c q^{- \max \left\{k + l - m - 1, 0 \right\}} / (a b x))_{\infty}} 
\tilde{P} \c{k - 1}{l - 1}{m - 1}{a q}{b q}{c q}{x}. \nonumber 
\end{align}
From $(\ref{tR,Z}),\ (\ref{Z(111)})$ and $(\ref{Z'})$, we have the expression for $\tilde{R}$. 
Thus the lemma is proved. 
\end{proof}
From $(\ref{Q,tQ}),\ (\ref{R,tR})$ and Lemma~$\ref{tQ,tR-2}$, 
we can complete the proof of Theorem~$\ref{main2}$.

\subsection{Proofs of Corollaries~$\ref{saikouji}$ and $\ref{R=Q'}$}\mbox{} \\

By using Theorems~$\ref{uniqueness}\mbox{--}\ref{main2}$, we show that Corollaries~$\ref{saikouji}$ and $\ref{R=Q'}$. 

First, we prove Corollary~$\ref{saikouji}$. 
From Theorem~$\ref{uniqueness}$, we find that $Q$ in Theorem~$\ref{main}$ and 
$Q$ in Theorem~$\ref{main2}$ are equal to each other. 
Therefore, by comparing the expressions for $Q$ in Theorems~$\ref{main}$ and $\ref{main2}$, 
we obtain $(\ref{P,til(P)})$. 
Also, we can easily verify that the degree of $P$ equals the degree of $\tilde{P}$ from $(\ref{P,til(P)})$. 
Let $d := \max \left\{m, 0 \right\} + \max \left\{k + l - m, 0 \right\} - k - 1$. Then, we have 
\begin{align*}
d = 
\begin{cases}
l - 1 & (m \geq 0,\, k + l - m \geq 0), \\
m - k - 1 & (m \geq 0,\, k + l - m < 0), \\
l - m - 1 & (m < 0,\, k + l - m \geq 0), \\
- k - 1 & (m < 0,\, k + l - m < 0). 
\end{cases}
\end{align*}
It follows from Definition~$\ref{def:til(P),C,D}$ that 
for any integers $k,\ l$ and $m$ with $k \leq l$, 
the degree of the polynomial $\tilde{P}$ is no more than $d$, 
and the coefficient of $x^{d}$ in $\tilde{P}$ 
equals $(C_{0} - D_{0}) \rvert_{k = l}$ when $k = l$, and equals $C_{0}$ when $k < l$ ; 
Namely, the coefficient of $x^{d}$ in $\tilde{P}$ equals 
\begin{align*}
\begin{cases}
\, \displaystyle\left(\frac{c}{a b} \right)^{k} q^{k (m - k - l + 1)} 
\left\{a^{m - 2k} (a)_{k} (a q / c)_{k - m} - b^{m - 2k} (b)_{k} (b q / c)_{k - m} \right\} & (k = l), \\
\, \displaystyle\left(\frac{c}{a b} \right)^{k} q^{k (m - k - l + 1)} a^{m - k - l} 
\frac{(a)_{k} (a q / c)_{k - m}}{(a q / b)_{k - l}} & (k < l). 
\end{cases}
\end{align*}
This completes the proof of Corollary~$\ref{saikouji}$. 

Next, we prove Corollary~$\ref{R=Q'}$. 
From Theorem~$\ref{main}$, we have 
\begin{align*}
&Q (k - 1, l - 1, m - 1) \rvert _{(a, b, c) \mapsto (a q, b q, c q)} \\
&\quad = - \frac{(1 - a q) (1 - b q) c}{(1 - c) (1 - c q)} 
\frac{x^{1 - \max \left\{m - 1, 0 \right\}}}{(a b q^{2} x / c)_{\max \left\{k + l - m - 1, 0 \right\} - 1}} 
P \c{k - 1}{l - 1}{m - 1}{a q}{b q}{c q}{x} \\
&\quad = \frac{(1 - a q) (1 - b q) x (c - a b q x)}{(1 - c) (1 - c q)} R (k, l, m). 
\end{align*}
Thus Corollary~$\ref{R=Q'}$ is proved. 
We can also obtain Corollary~$\ref{R=Q'}$ from Theorem~$\ref{main2}$ in the same way. 

\thanks{\bf{Acknowledgements}}
We are deeply grateful to Prof. Hiroyuki Ochiai for helpful comments. 
Also, we would like to thank Akihito Ebisu for his comments and suggestions. 
Furthermore, thanks to our colleagues in the Graduate School of Mathematics of Kyushu University.

\medskip
\begin{flushleft}
Yuka Suzuki\\
Graduate School of Mathematics\\
Kyushu University\\
Nishi-ku, Fukuoka 819-0395 \\
Japan\\
y-suzuki@math.kyushu-u.ac.jp
\end{flushleft}

\end{document}